\documentclass{article}
\usepackage{amsmath}
\usepackage{amsfonts}

\newtheorem{theorem}{Theorem}
\newtheorem{corollary}{Corollary}
\newtheorem{lemma}{Lemma}
\newtheorem{remark}{Remark}
\newtheorem{conjecture}{Conjecture}

\newtheorem{proposition}{Proposition}
\newenvironment{proof}[1][Proof]{\noindent\textbf{#1.} }{\ \rule{0.5em}{0.5em}}
\input{tcilatex}
\begin{document}

\title{Means and nonreal Intersection Points of Taylor Polynomials}
\author{Alan Horwitz \\
Professor Emeritus of Mathematics\\
Penn State Brandywine\\
25 Yearsley Mill Rd.\\
Media, PA 19063\\
alh4@psu.edu}
\date{2/4/15}
\maketitle

\begin{abstract}
Suppose that $f\in C^{r+1}(0,\infty )$, and let $P_{c}$ denote the Taylor
polynomial to $f$ of order $r$ at $x=c\in \lbrack a,b]$. In \cite{H1} it was
shown that if $r$ is an odd whole number and $f^{(r+1)}(x)\neq 0$ on $[a,b]$%
, then there is a unique $x_{0},a<x_{0}<b$, such that $%
P_{a}(x_{0})=P_{b}(x_{0})$. This defines a mean $M_{f}^{r}(a,b)\equiv x_{0}$%
. In this paper we discuss the \textbf{real parts} of the pairs of complex
conjugate \textbf{nonreal} roots of $P_{b}-P_{a}$. We prove some results for 
$r$ in general, but our most significant results are for the case $r=3$. We
prove in that case that if $f(z)=z^{p}$, where $p$ is an \textbf{integer}, $%
p\notin \left\{ 0,1,2,3\right\} $, then $P_{b}-P_{a}$ has nonreal roots $%
x_{1}\pm iy_{1}$, with $a<x_{1}<b$ for any $0<a<b$. This defines the
countable family of means $M_{z^{p}}^{3}(a,b)$, where $p=n\in 
\mathbb{Z}
-\left\{ 0,1,2,3\right\} $. We construct a cubic polynomial, $g$, whose real
root gives the real part of the pair of complex conjugate nonreal roots of $%
P_{b}-P_{a}$. Instead of working directly with a formula for the roots of a
cubic, we use the Intermediate Value Theorem to show that $g$ has a root in $%
(a,b)$.

\end{abstract}

\section{Introduction}

Suppose that $f\in C^{r+1}(0,\infty )$, and let $P_{c}$ denote the Taylor
polynomial to $f$ of order $r$ at $x=c>0$. In (\cite{H1}, Theorem 1.1) it
was proved that if $r$ is an odd whole number and $f^{(r+1)}(x)\neq 0$ on $%
[a,b]$, $0<a<b$, then there is a unique real number $x_{0},a<x_{0}<b$, such
that $P_{a}(x_{0})=P_{b}(x_{0})$. This, of course, defines a mean $%
M_{f}^{r}(a,b)\equiv x_{0}$. Further results and generalizations of the
means $M_{f}^{r}(a,b)$ were proved in \cite{H2}, where $r$ is any positive
integer, odd or even. The main purpose of this paper is to discuss the 
\textbf{real parts} of the pairs of complex conjugate \textbf{nonreal} roots
of $P_{b}-P_{a}$. In Proposition \ref{P1} below we show, for any \textit{odd}
whole number, $r$, and under suitable assumptions on $f$, that $P_{b}-P_{a}$
has precisely one real zero $x_{0},a<x_{0}<b$. We also show that for any 
\textit{even} whole number, $r$, $P_{b}-P_{a}$ has all nonreal zeros. The
main question is then:

What can we say about the real parts of the pairs of complex conjugate
nonreal roots of $P_{b}-P_{a}$ ? In particular, when do the real parts lie
strictly between $a$ and $b$ ? That is, what conditions on $f$ imply that
the real parts of the nonreal roots of $P_{b}-P_{a}$ define a mean ? We
cannot answer that question completely for $r$ in general, but we are able
to prove some results in \S\ \ref{S1}. If $f^{(r+1)}(x)$ is continuous and
has no zeros in $[a,b]$, then the \textbf{averages} of the pairs of complex
conjugate nonreal roots of $P_{b}-P_{a}$ lie strictly between $a$ and $b$
for any positive integer $r$(Proposition \ref{P2}). This, of course, does
not tell us what happens with the real parts of each \textbf{specific}
nonreal root. However, for $r=2$ one gets immediately that if $f\,^{\prime
\prime \prime }(x)\neq 0$ on $[a,b]$, and if $x_{1}\pm iy_{1}$ are the
nonreal roots of $P_{b}-P_{a}$, then $a<x_{1}<b$ for any $0<a<b$(Corollary %
\ref{C1} ). Also, if $f(z)=z^{r+1}$, then all of the nonreal roots of $%
P_{b}-P_{a}$ have real part given by the arithmetic mean $A(a,b)=\dfrac{a+b}{%
2}$(Proposition \ref{P4} ). Most of the more detailed and complex proofs in
this paper are in \S\ \ref{S2}, which involves the case $r=3$. We
prove(Theorem \ref{T1}) that if $f(z)=z^{p}$, where $p$ is an \textbf{integer%
}, $p\notin \left\{ 0,1,2,3\right\} $, then $P_{b}-P_{a}$ has nonreal roots $%
x_{1}\pm iy_{1}$, with $a<x_{1}<b$ for any $0<a<b$. This defines the
countable family of means $M_{z^{p}}^{3}(a,b)$, where $p\in 
\mathbb{Z}
-\left\{ 0,1,2,3\right\} $. For example, if $p=-1$, one has $x_{1}=\dfrac{%
ab(a+b)}{a^{2}+b^{2}}$. We have proven some partial results when $p$ is 
\textbf{not} an integer, but Theorem \ref{T1} does \textbf{not }hold in
general for $p\in \Re $, $p\notin \left\{ 0,1,2,3\right\} $. For example,
let $p=\dfrac{3}{2}$, $a=1$, and $b=36$. Then $P_{36}(z)-P_{1}(z)$ has roots 
$6,\dfrac{33}{43}\pm \dfrac{15}{43}\sqrt{291}i$, so that $x_{0}<a$. In \S\ %
\ref{S3} we give some alternate proofs and some partial results, which are
perhaps interesting in their own right, and which also might lead to
determining what conditions on $f$ imply that the real parts of the pairs of
complex conjugate nonreal roots of $P_{b}-P_{a}$ lie in $(a,b)$. Finally we
consider possible future research in \S\ \ref{S4}.

\section{General $r$\label{S1}}

A result very similar to Proposition \ref{P1}(i) below was proved in \cite%
{H3}. Since the proof is short and we need some functions from that proof to
prove Proposition \ref{P1}(ii) as well as for some later material, we give
the full proof here. Assume that $0<a<b$ and that all functions $f$ are real
valued for the rest of the paper.

\begin{proposition}
\label{P1}Suppose that $f^{(r+1)}$ is continuous in some open interval
containing $[a,b]$ and has no zeros in $[a,b]$. Let $P_{c}$ denote the
Taylor polynomial to $f$ of order $r$ at $x=c$.

(i) If $r$ is an \textit{odd} positive integer, then $P_{b}-P_{a}$ has
precisely one real zero $x_{0},a<x_{0}<b$.

(ii) If $r$ is an \textit{even} positive integer then $P_{b}-P_{a}$ has all
nonreal zeros.
\end{proposition}

\begin{proof}
Let $E_{c}(x)=f(x)-P_{c}(x),x\in \lbrack a,b]$. Note that $%
P_{b}(x)=P_{a}(x)\iff E_{b}(x)=E_{a}(x)$. By the integral form of the
remainder, we have \newline
$E_{c}(x)=\dfrac{1}{r!}\dint\limits_{c}^{x}f^{(r+1)}(t)(x-t)^{r}dt$, which
implies that 
\begin{eqnarray*}
E_{a}(x) &=&\dfrac{1}{r!}\dint\limits_{a}^{x}f^{(r+1)}(t)(x-t)^{r}dt \\
E_{b}(x) &=&-\dfrac{1}{r!}\dint\limits_{x}^{b}f^{(r+1)}(t)(x-t)^{r}dt\text{.}
\end{eqnarray*}

Since $(E_{a}-E_{b})(x)=\dfrac{1}{r!}\dint%
\limits_{a}^{b}f^{(r+1)}(t)(x-t)^{r}dt$ and $%
E_{a}(x)-E_{b}(x)=P_{b}(x)-P_{a}(x),x\in \lbrack a,b]$, we have 
\begin{equation}
(P_{b}-P_{a})(x)=\dfrac{1}{r!}\dint\limits_{a}^{b}f^{(r+1)}(t)(x-t)^{r}dt%
\text{.}  \label{Pab}
\end{equation}%
Since (\ref{Pab}) holds for $x\in \lbrack a,b]$ and $P_{b}-P_{a}$ is a
polynomial, (\ref{Pab}) holds for all $z\in 
\mathbb{C}
=$ complex plane. We may assume, without loss of generality, that $%
f^{(r+1)}(x)>0$ on $[a,b]$. Suppose first that $r$ is odd. Using the formula 
$\dfrac{\partial }{\partial x}\dint\limits_{a}^{x}K(x,t)dt=\dint%
\limits_{a}^{x}\dfrac{\partial K}{\partial x}(x,t)dt+K(x,x)$ with $K(x,t)=$ $%
f^{(r+1)}(t)(x-t)^{r-1}$, it follows that \newline
$E_{a}^{\prime }(x)=\dfrac{1}{(r-1)!}\dint%
\limits_{a}^{x}f^{(r+1)}(t)(x-t)^{r-1}dt>0$ for $a<x<b$, which implies that $%
E_{a}(x)$ is strictly increasing on $(a,b)$.

$E_{b}^{\prime }(x)=-\dfrac{1}{(r-1)!}\dint%
\limits_{x}^{b}f^{(r+1)}(t)(x-t)^{r-1}dt<0$ for $a<x<b$, which implies that $%
E_{b}(x)$ is strictly decreasing on $(a,b)$. Since $E_{a}(a)=E_{b}(b)=0$,
there is a unique $x_{0},a<x_{0}<b$, such that $E_{b}(x_{0})-E_{a}(x_{0})=0$%
, which implies that $(P_{b}-P_{a})(x_{0})=0$. Now $(P_{b}-P_{a})^{\prime
}(x)=\dfrac{1}{(r-1)!}\dint\limits_{a}^{b}f^{(r+1)}(t)(x-t)^{r-1}dt>0$ for $%
x\in \Re =$ real line. Hence $P_{b}-P_{a}\;$has precisely one real zero.
That proves (i).

Now suppose that $r$ is even. Then $(P_{b}-P_{a})(x)>0$ for $x\in \Re $,
which implies that $P_{b}-P_{a}$ has no real zeros. That proves (ii).
\end{proof}

(\ref{Pab}) gives the following important equivalence: If $P_{c}$ is the
Taylor polynomial to $f$ of order $r$ at $x=c$, then

\begin{equation}
P_{b}(z)=P_{a}(z)\iff \dint\limits_{a}^{b}f^{(r+1)}(t)(z-t)^{r}dt=0,z\in 
\mathbb{C}
\text{.}  \label{error}
\end{equation}

\qquad We now prove a general result which relates the \textbf{averages} of
the real parts of the roots of $P_{b}-P_{a}$ to the center of mass of $[a,b]$
with density function $\left\vert f\,^{(r+1)}(t)\right\vert $.

\begin{proposition}
\label{P2}Suppose that $f^{(r+1)}$ is continuous in some open interval
containing $[a,b]$ and has no zeros in $[a,b]$. Let $P_{c}$ denote the
Taylor polynomial to $f$ of order $r$ at $x=c$.

(i) Suppose that $r$ is odd and let $s=\dfrac{r-1}{2}$. By Proposition \ref%
{P1}(ii), $P_{b}-P_{a}$ has precisely one real zero $x_{0}$, and $r-1$
nonreal zeros, $z_{1},\bar{z}_{1},...,z_{s},\bar{z}_{s}$. Let $x_{k}=\func{Re%
}z_{k},k=1,...,s$. Then 
\begin{equation*}
\dfrac{x_{0}+\dsum\limits_{k=1}^{s}2\func{Re}z_{k}}{r}=\dfrac{%
\dint\limits_{a}^{b}tf\,^{(r+1)}(t)dt}{\dint\limits_{a}^{b}f\,^{(r+1)}(t)dt}.
\end{equation*}

(ii) Suppose that $r$ is even and let $s=\dfrac{r}{2}$. By Proposition \ref%
{P1}(ii), $P_{b}-P_{a}$ has all nonreal zeros, $z_{1},\bar{z}_{1},...,z_{s},%
\bar{z}_{s}$. Let $x_{k}=\func{Re}z_{k},k=1,...,s$. Then 
\begin{equation*}
\dfrac{\dsum\limits_{k=1}^{s}2\func{Re}z_{k}}{r}=\dfrac{\dint%
\limits_{a}^{b}tf\,^{(r+1)}(t)dt}{\dint\limits_{a}^{b}f\,^{(r+1)}(t)dt}\text{%
.}
\end{equation*}
\end{proposition}

In either case, since $P_{b}-P_{a}$ has real coefficients, Proposition \ref%
{P2} states that the average of the real parts of the roots of $P_{b}-P_{a}$
is the center of mass of $[a,b]$, where the density function is $\left\vert
f\,^{(r+1)}(t)\right\vert $.

\begin{proof}
We may assume, without loss of generality, that $f^{(r+1)}(x)>0$ on $[a,b]$.
Define the monic polynomial $Q(z)=r!\dfrac{\left( P_{b}-P_{a}\right) (z)}{%
f^{(r)}(b)-f^{(r)}(a)}$. Since $\left( P_{b}-P_{a}\right)
(z)=\dsum\limits_{k=0}^{r}\dfrac{f^{(k)}(b)(z-b)^{k}-f^{(k)}(a)(z-a)^{k}}{k!}
$, the coefficient of $z^{r-1}$ in $Q(z)$ is 
\begin{gather*}
\dfrac{{\large (}rf^{(r-1)}(b)-rbf^{(r)}(b){\large )}-{\large (}%
rf^{(r-1)}(a)-raf^{(r)}(a){\large )}}{f^{(r)}(b)-f^{(r)}(a)}= \\
r\dfrac{{\large (}f^{(r-1)}(b)-bf^{(r)}(b){\large )}-{\large (}%
f^{(r-1)}(a)-af^{(r)}(a){\large )}}{f^{(r)}(b)-f^{(r)}(a)}\text{.}
\end{gather*}%
Using Integration by Parts, it is easy to show that

\begin{equation*}
\dfrac{\dint\limits_{a}^{b}tf\,^{(r+1)}(t)dt}{\dint\limits_{a}^{b}f%
\,^{(r+1)}(t)dt}=-\dfrac{{\large (}f^{(r-1)}(b)-bf^{(r)}(b){\large )}-%
{\large (}f^{(r-1)}(a)-af^{(r)}(a){\large )}}{f^{(r)}(b)-f^{(r)}(a)}\text{.}
\end{equation*}%
Thus the coefficient of $z^{r-1}$ in $Q(z)$ is $-r\dfrac{\dint%
\limits_{a}^{b}tf\,^{(r+1)}(t)dt}{\dint\limits_{a}^{b}f\,^{(r+1)}(t)dt}$.

Now suppose that $r$ is odd. Since\ $P_{b}-P_{a}$ and $Q$ have the same
roots, $Q(z)=(z-x_{0})(z-z_{1})(z-\bar{z}_{1})\cdots (z-z_{s})(z-\bar{z}%
_{s}) $. Since the coefficient of $z^{r-1}$ in $Q(z)$ is also given by

\begin{equation*}
-x_{0}-\left( z_{1}+\bar{z}_{1}+\cdots +z_{s}+\bar{z}_{s}\right)
=-x_{0}-2\dsum\limits_{k=1}^{s}x_{k}\text{,}
\end{equation*}%
we have $-x_{0}-2\dsum\limits_{k=1}^{s}x_{k}=-r\dfrac{\dint\limits_{a}^{b}tf%
\,^{(r+1)}(t)dt}{\dint\limits_{a}^{b}f\,^{(r+1)}(t)dt}$, which implies that

$\dfrac{x_{0}+\dsum\limits_{k=1}^{s}2\func{Re}z_{k}}{r}=\dfrac{%
\dint\limits_{a}^{b}tf\,^{(r+1)}(t)dt}{\dint\limits_{a}^{b}f\,^{(r+1)}(t)dt}$%
.

Now suppose that $r$ is even and write $Q(z)=(z-z_{1})(z-\bar{z}_{1})\cdots
(z-z_{s})(z-\bar{z}_{s})$. Since the coefficient of $z^{r-1}$ in $Q(z)$ is
also given by%
\begin{equation*}
-\left( z_{1}+\bar{z}_{1}+\cdots +z_{s}+\bar{z}_{s}\right)
=-2\dsum\limits_{k=1}^{s}x_{k}\text{,}
\end{equation*}%
we have

$-2\dsum\limits_{k=1}^{s}x_{k}=-r\dfrac{\dint\limits_{a}^{b}tf\,^{(r+1)}(t)dt%
}{\dint\limits_{a}^{b}f\,^{(r+1)}(t)dt}$, which implies that

$\dfrac{\dsum\limits_{k=1}^{s}2\func{Re}z_{k}}{r}=\dfrac{\dint%
\limits_{a}^{b}tf\,^{(r+1)}(t)dt}{\dint\limits_{a}^{b}f\,^{(r+1)}(t)dt}$.
\end{proof}

\textbf{Remark: }Suppose that $f^{(r+1)}$ is continuous and has no zeros in $%
[a,b]$. If\ the average\ of the real parts of the nonreal roots of $%
P_{b}-P_{a}$ did not lie in $(a,b)$, then it would not be possible for the
real parts of all of the nonreal roots of $P_{b}-P_{a}$ to lie in $(a,b)$.
For each $r$, there are examples, such as $f(z)=z^{r+1}$(see Proposition \ref%
{P4} below), where the real parts of all pairs of complex conjugate roots of 
$P_{b}-P_{a}$ do lie in $(a,b)$ for all $0<a<b$. Of course if $r=3$, then
there is only one pair of nonreal complex conjugate roots. As noted above,
it is possible that the real parts of that complex conjugate pair do not lie
in $(a,b)$. For example, $f(z)=z^{3/2}$, $a=1$, and $b=36$. We also have
examples for $r=4$ and for $r=5$ where only one pair of nonreal complex
conjugate roots has real part lying in $(a,b)$.

If $r=2$, then $P_{b}-P_{a}$ has no real roots and only one pair of complex
conjugate nonreal roots. Applying Proposition \ref{P2}(ii) then yields the
following corollary.

\begin{corollary}
\label{C1}Suppose that $f\,^{\prime \prime \prime }$ is continuous in some
open interval containing $[a,b]$ and has no zeros in $[a,b]$. Let $P_{c}$
denote the Taylor polynomial to $f$ of order $2$ at $x=c$. Let $x_{1}\pm
iy_{1}$ denote the nonreal roots of $P_{b}-P_{a}$ guaranteed by Proposition %
\ref{P1}(ii). Then 
\begin{equation*}
x_{1}=\dfrac{\dint\limits_{a}^{b}tf\,^{\prime \prime \prime }(t)dt}{%
\dint\limits_{a}^{b}f\,^{\prime \prime \prime }(t)dt}
\end{equation*}%
and thus $a<x_{1}<b$.
\end{corollary}

The case $r=1$ of Proposition \ref{P2} above was given in \cite{H1}, where $%
x_{1}=\dfrac{\dint\limits_{a}^{b}tf\,^{\prime \prime }(t)dt}{%
\dint\limits_{a}^{b}f\,^{\prime \prime }(t)dt}$ is just the $x$ coordinate
of the intersection point of tangent lines to a convex or concave function.
This was actually the starting point that led to the generalization to
intersection points of Taylor polynomials. Replacing $f\,^{\prime \prime
\prime }$ by $f\,^{\prime \prime }$ shows that this yields the same family
of means as for the $r=2$ case.

\begin{proposition}
\label{P4}Let $f(z)=z^{r+1}$ and let $P_{c}$ denote the Taylor polynomial to 
$f$ of order $r$ at $x=c$. Then all of the nonreal roots of $P_{b}-P_{a}$
have real part given by the arithmetic mean $A(a,b)=\dfrac{a+b}{2}$.
\end{proposition}

\begin{proof}
By (\ref{error}), it suffices to show that $\dint%
\limits_{a}^{b}f^{(r+1)}(t)(z-t)^{r}dt=0\iff \func{Re}z=\dfrac{a+b}{2}$. Now 
$\dint\limits_{a}^{b}f^{(r+1)}(t)(z-t)^{r}dt=0\iff
\dint\limits_{a}^{b}(z-t)^{r}dt=0\iff (z-b)^{r+1}-(z-a)^{r+1}=0\iff
z-b=v(z-a)$, where $v=e^{2\pi ki/(r+1)}$ is an $(r+1)$st root of unity. Note
that $v\neq 1$ since $a\neq b$. Thus $z=\dfrac{b-va}{1-v}=\dfrac{-v}{1-v}a+%
\dfrac{1}{1-v}b$. 
\begin{gather*}
\dfrac{1}{1-v}=\dfrac{1-\bar{v}}{(1-v)(1-\bar{v})} \\
=\dfrac{1-\bar{v}}{1-2\func{Re}v+\left\vert v\right\vert ^{2}}=\dfrac{1}{2(1-%
\func{Re}v)}(1-\bar{v})\text{,}
\end{gather*}%
therefore $\func{Re}\left( \dfrac{1}{1-v}\right) =\dfrac{1}{2(1-\func{Re}v)}%
(1-\func{Re}v)=\dfrac{1}{2}$.

$1=\dfrac{-v}{1-v}+\dfrac{1}{1-v}$ implies that $\func{Re}\left( \dfrac{-v}{%
1-v}\right) =1-\func{Re}\left( \dfrac{1}{1-v}\right) =\dfrac{1}{2}$. Hence $%
\func{Re}z=\dfrac{1}{2}a+\dfrac{1}{2}b$.
\end{proof}

\section{$r=3$\label{S2}}

\qquad We now state our main result for $r=3$.

\begin{theorem}
\label{T1}Suppose that $f(z)=z^{p}$, where $p$ is an integer, $p\notin
\left\{ 0,1,2,3\right\} $. Let $P_{c}$ denote the Taylor polynomial to $f$
of order $3$ at $x=c$. Then for any $0<a<b$, $P_{b}-P_{a}$ has nonreal roots 
$x_{1}\pm iy_{1}$, with $a<x_{1}<b$.
\end{theorem}

\textbf{Remark: }Theorem \ref{T1} defines a countable family of means $%
M_{z^{p}}^{3}(a,b)=x_{1}$, where $p\in 
\mathbb{Z}
-\left\{ 0,1,2,3\right\} $. By Proposition \ref{P4} with $r=3$, amongst that
family of means is the arithmetic mean.

\textbf{Remark: }Finding $z_{1}$ such that $P_{b}(z_{1})=P_{a}(z_{1})$ of
course involves solving a cubic polynomial equation. There are well--known
formulas for the solutions of such equations, but the resulting expressions
are complicated and it seems difficult to determine from such a formula that 
$a<x_{1}<b$. For example, if $p=5$ and $a=1$, then one has 
\begin{eqnarray*}
x_{1} &=&\dfrac{1}{3}\dfrac{b-1}{b+1}\times \\
&&{\large (}\tfrac{1}{20}\sqrt[3]{100\left( b^{2}+7b+1\right) \left(
b-1\right) +150\sqrt{6}(b+1)\sqrt{q(b)}} \\
&&-\dfrac{5}{2}\dfrac{b^{2}+4b+1}{\sqrt[3]{100\left( b^{2}+7b+1\right)
\left( b-1\right) +150\sqrt{6}(b+1)\sqrt{q(b)}}} \\
&&+2\left( b^{2}+b+1\right) \text{,}
\end{eqnarray*}%
where $q(b)=b^{4}+10b^{3}+28b^{2}+10b+1$. Furthermore, we want to determine
that for certain classes of functions, $f$, $a<\func{Re}(z_{1})<b$. Our
proof of Theorem \ref{T1} also involves solving a certain cubic polynomial
equation, $g(x)=0$(see (\ref{g}) below). However, this time we are looking
for a \textbf{real} solution, $x_{1}$, of $g(x)=0$, with $a<x_{1}<b$. That
allows us to use the Intermediate Value Theorem to show that there is such a
solution. That is, we show that for certain classes of functions, $f$, $%
g(a)g(b)<0$. That avoids actually working with a formula for the solution of
a cubic polynomial equation.

If $P_{c}$ is the Taylor polynomial to $f$ of order $3$ at $x=c$, then (\ref%
{error}) becomes 
\begin{equation}
P_{b}(z)=P_{a}(z)\iff \dint\limits_{a}^{b}f\,^{\prime \prime \prime \prime
}(t)(z-t)^{3}dt=0\text{.}  \label{1}
\end{equation}

For the rest of this section we prove some lemmas and propositions which are
used to prove Theorem \ref{T1}. Important for our proofs are the following
integrals. Let 
\begin{equation}
A=\dint\limits_{a}^{b}f\,^{\prime \prime \prime \prime
}(t)dt,B=\dint\limits_{a}^{b}tf\,^{\prime \prime \prime \prime
}(t)dt,C=\dint\limits_{a}^{b}t^{2}f\,^{\prime \prime \prime \prime
}(t)dt,D=\dint\limits_{a}^{b}t^{3}f\,^{\prime \prime \prime \prime }(t)dt%
\text{.}  \label{ABCD}
\end{equation}%
We suppress the dependence of $A,B,C$, and $D$ on $a,b$, and on $f$ in our
notation. We now prove a lemma which gives an equivalent condition for $%
P_{b}(z_{1})=P_{a}(z_{1})$ to hold when $r=3$.

\begin{lemma}
\label{L2}Let $P_{c}$ denote the Taylor polynomial to $f$ of order $3$ at $%
x=c$, and let $z_{1}=x_{1}+iy_{1}$ with $y_{1}\neq 0$. Then $%
P_{b}(z_{1})=P_{a}(z_{1})$ if and only if the following system of equations
holds.

\begin{eqnarray}
Ax_{1}^{3}-3Bx_{1}^{2}+3Cx_{1}-D+3(B-Ax_{1})y_{1}^{2} &=&0  \label{sys2} \\
3Ax_{1}^{2}-6Bx_{1}+3C-Ay_{1}^{2} &=&0\text{.}  \notag
\end{eqnarray}
\end{lemma}

\begin{proof}
Using the formulas 
\begin{eqnarray*}
\func{Re}{\large (}(z_{1}-t)^{3}{\large )} &=&{\large (}\func{Re}(z_{1})-t%
{\large )}^{3}-3{\large (}\func{Re}(z_{1})-t{\large )}\func{Im}^{2}\left(
z_{1}\right) \\
\func{Im}{\large (}(z_{1}-t)^{3}{\large )} &=&3{\large (}\func{Re}(z_{1})-t%
{\large )}^{2}\func{Im}\left( z_{1}\right) -{\large (}\func{Im}\left(
z_{1}\right) {\large )}^{3}\text{,}
\end{eqnarray*}%
we have 
\begin{equation*}
\func{Re}\left( \dint\limits_{a}^{b}f\,^{\prime \prime \prime \prime
}(t)(z_{1}-t)^{3}dt\right) =\dint\limits_{a}^{b}f\,^{\prime \prime \prime
\prime }(t)\left[ {\large (}x_{1}-t{\large )}^{3}-3{\large (}x_{1}-t{\large )%
}y_{1}^{2}\right] dt
\end{equation*}%
and 
\begin{equation*}
\func{Im}\left( \dint\limits_{a}^{b}f\,^{\prime \prime \prime \prime
}(t)(z_{1}-t)^{3}dt\right) =\dint\limits_{a}^{b}f\,^{\prime \prime \prime
\prime }(t)\left[ 3{\large (}x_{1}-t{\large )}^{2}y_{1}-y_{1}^{3}\right] dt.
\end{equation*}%
$P_{b}(z_{1})=P_{a}(z_{1})\iff \dint\limits_{a}^{b}f\,^{\prime \prime \prime
\prime }(t)(z_{1}-t)^{3}dt=0$ by (\ref{1}). If $y_{1}\neq 0$, then $%
\dint\limits_{a}^{b}f\,^{\prime \prime \prime \prime }(t)(z_{1}-t)^{3}dt=0$
is equivalent to the following two equations: 
\begin{eqnarray}
\dint\limits_{a}^{b}f\,^{\prime \prime \prime \prime }(t)\left[ {\large (}%
x_{1}-t{\large )}^{3}-3{\large (}x_{1}-t{\large )}y_{1}^{2}\right] dt &=&0
\label{sys} \\
\dint\limits_{a}^{b}f\,^{\prime \prime \prime \prime }(t)\left[ 3{\large (}%
x_{1}-t{\large )}^{2}-y_{1}^{2}\right] dt &=&0\text{.}  \notag
\end{eqnarray}%
Simplifying (\ref{sys}) shows that $(x_{1},y_{1})$ satisfies (\ref{sys}) if
and only if $(x_{1},y_{1})$ satisfies (\ref{sys2}).
\end{proof}

We now define the following very important cubic polynomial, $g$, which
depends upon the given function, $f$, as well as on $a$ and $b$:

\begin{equation}
g(x)=8A^{2}x^{3}-24ABx^{2}+6(AC+3B^{2})x+AD-9BC\text{,}  \label{g}
\end{equation}%
where $A,B,C$, and $D$ are given by (\ref{ABCD}).

\begin{lemma}
\label{L1}Let $A\neq 0,B,C\in \Re $ and not necessarily given by (\ref{ABCD}%
). If $B^{2}-AC<0$, then $g$ is increasing on $\Re $.
\end{lemma}

\begin{proof}
$g^{\prime }(x)=\allowbreak 24A^{2}x^{2}-48ABx+6AC+18B^{2}=$

$24A^{2}\left( x^{2}-\dfrac{2B}{A}x+\dfrac{AC+3B^{2}}{4A^{2}}\right)
=24A^{2}\left( \left( x-\dfrac{B}{A}\right) ^{2}+\dfrac{AC-B^{2}}{4A^{2}}%
\right) >0$
\end{proof}

\begin{remark}
\label{R2}Let $A\neq 0,B,C\in \Re $. If $B^{2}-AC<0$ and $x_{1}\in \Re $,
then

$3x_{1}^{2}-\dfrac{6B}{A}x_{1}+\dfrac{3C}{A}=3\left( \left( x_{1}-\dfrac{B}{A%
}\right) ^{2}+\dfrac{AC-B^{2}}{A^{2}}\right) >0$.
\end{remark}

Critical for our proof of Theorem \ref{T1} below is the following
proposition.

\begin{proposition}
\label{P3}Suppose that $f\,^{\prime \prime \prime \prime }$ is continuous in
some open interval containing $[a,b]$ and has no zeros in $[a,b]$, and let $%
P_{c}$ denote the Taylor polynomial to $f$ of order $3$ at $x=c$. Then the
polynomial \ $g$ given by (\ref{g}) has a unique real zero. In addition, if $%
x_{1}\in \Re ,y_{1}=\sqrt{3x_{1}^{2}-\dfrac{6B}{A}x_{1}+\dfrac{3C}{A}}$, and 
$z_{1}=x_{1}+iy_{1}$, then $g(x_{1})=0\iff P_{b}(z_{1})=P_{a}(z_{1})$.
\end{proposition}

\begin{proof}
If $f\,^{\prime \prime \prime \prime }>0$ on $[a,b]$, then $\left(
\dint\limits_{a}^{b}tf\,^{\prime \prime \prime \prime }(t)dt\right)
^{2}=\left( \dint\limits_{a}^{b}\sqrt{f\,^{\prime \prime \prime \prime }(t)}%
\left( t\sqrt{f\,^{\prime \prime \prime \prime }(t)}\right) dt\right)
^{2}<\left( \dint\limits_{a}^{b}f\,^{\prime \prime \prime \prime
}(t)dt\right) \left( \dint\limits_{a}^{b}t^{2}f\,^{\prime \prime \prime
\prime }(t)dt\right) $ by the Cauchy-Bunyakowsky inequality(\cite{S}). Note
that the strict inequality follows since $\sqrt{f\,^{\prime \prime \prime
\prime }(t)}$ and $t\sqrt{f\,^{\prime \prime \prime \prime }(t)}$ cannot be
proportional to one another. Since $B^{2}-AC$ does not depend on the sign of 
$f\,^{\prime \prime \prime \prime }$, where $A,B,C$ are given by (\ref{ABCD}%
), we have $B^{2}-AC<0$ when $f\,^{\prime \prime \prime \prime }\neq 0$ on $%
[a,b]$. Thus $g$ has a unique real zero by Lemma \ref{L1} and the fact that $%
g$ is a cubic polynomial. Note that by Remark \ref{R2}, $y_{1}$ is real and
positive. By Lemma \ref{L2}, $P_{b}(z_{1})=P_{a}(z_{1})\iff $ (\ref{sys2})
holds. Since $A\neq 0$, solving the second equation in (\ref{sys2}) for $%
y_{1}^{2}$ and substituting into the first equation in (\ref{sys2}) to
obtain $y_{1}^{2}=3x_{1}^{2}-\dfrac{6B}{A}x_{1}+\dfrac{3C}{A}$shows that (%
\ref{sys2}) holds if and only if $%
Ax_{1}^{3}-3Bx_{1}^{2}+3Cx_{1}-D+3(B-Ax_{1})\left( 3x_{1}^{2}-\dfrac{6B}{A}%
x_{1}+\dfrac{3C}{A}\right) =0\iff $

$(A-9A)x_{1}^{3}+(-3B+18B+9B)x_{1}^{2}+\left( 3C-9C-18\dfrac{B^{2}}{A}%
\right) x_{1}-D+\dfrac{9BC}{A}=0\iff $

$-8Ax_{1}^{3}+24Bx_{1}^{2}-6\left( C+3\dfrac{B^{2}}{A}\right) x_{1}+\dfrac{%
9BC}{A}-D=0\iff
-8A^{2}x_{1}^{3}+24ABx_{1}^{2}-6(AC+3B^{2})x_{1}+9BC-AD=0\iff g(x_{1})=0$.
\end{proof}

We now focus on the case where $f(z)=z^{p},p\in \Re -\left\{ 0,1,2,3\right\} 
$. For the purpose of proving Theorem \ref{T1}, it will suffice(as shown in
the proof below) to just consider the case when $a=1$, which we assume from
now on. For $f(z)=z^{p}$, (\ref{ABCD}) then yields

\begin{eqnarray}
A &=&\allowbreak p(p-1)(p-2){\large (}b^{p-3}-1{\large )},B=\allowbreak
p(p-1)(p-3){\large (}b^{p-2}-1{\large )},  \label{ABCDp} \\
C &=&\allowbreak p(p-2)(p-3){\large (}b^{p-1}-1{\large )},D=\left(
p-1\right) \left( p-2\right) \left( p-3\right) \left( b^{p}-1\right) \text{.}
\notag
\end{eqnarray}

\qquad Let $g(x)$ be given by (\ref{g}), where $A,B,C$, and $D$ are given by
(\ref{ABCDp}). For $p\in \Re ,p\neq 0,1$, it is more convenient to define
the following functions of $b$: 
\begin{eqnarray}
V(b) &=&\dfrac{g(1)}{p(p-1)}  \label{VW} \\
W(b) &=&\dfrac{g(b)}{p(p-1)}\text{.}  \notag
\end{eqnarray}

It is important to note that $g$ is a cubic polynomial in $x$ where the
coefficients involve $b$. $V$ and $W$ are functions of the variable $b$.

Using (\ref{ABCDp}) and substituting for $A,B,C$, and $D$ in (\ref{g}),
yields $g(1)=8A^{2}-24AB+6(AC+3B^{2})+AD-9BC=$

$8\left( \allowbreak p(p-1)(p-2)\right) ^{2}{\large (}b^{p-3}-1{\large )}%
^{2}-24\left( \allowbreak p(p-1)\right) ^{2}(p-2)(p-3){\large (}b^{p-3}-1%
{\large )(}b^{p-2}-1{\large )+}$

$6\left( \allowbreak p(p-2)\right) ^{2}(p-1)(p-3){\large (}b^{p-3}-1{\large %
)(}b^{p-1}-1{\large )}+18\left( \allowbreak p(p-1)(p-3)\right) ^{2}{\large (}%
b^{p-2}-1{\large )}^{2}+$

$\left( \left( p-1\right) \left( p-2\right) \right) ^{2}p\left( p-3\right) 
{\large (}b^{p-3}-1{\large )}\left( b^{p}-1\right) -9\left( \allowbreak
p(p-3)\right) ^{2}(p-1)(p-2){\large (}b^{p-2}-1{\large )(}b^{p-1}-1{\large )}
$,

which implies, after some simplification, that 
\begin{eqnarray}
V(b) &=&12(p+1)-2(p-2)(p-3)\left( 4p^{2}-12p-1\right) b^{2p-3}+  \notag \\
&&6p(p-3)\left( 4p^{2}-16p+13\right) b^{2p-4}-24p(p-1)(p-2)(p-3)b^{2p-5}+ 
\notag \\
&&8p(p-1)(p-2)^{2}b^{2p-6}-\left( p-1\right) (p-2)^{2}(p-3)b^{p}+  \label{V}
\\
&&3p(p-2)(p-3)\left( p-5\right) b^{p-1}-3p(p-3)\left( p^{2}-9p+2\right)
b^{p-2}+  \notag \\
&&(p-2)\left( p+1\right) \left( p^{2}-13p+6\right) b^{p-3}\text{.}  \notag
\end{eqnarray}

$g(b)=8A^{2}b^{3}-24ABb^{2}+6(AC+3B^{2})b+AD-9BC$

$=8\left( \allowbreak p(p-1)(p-2)\right) ^{2}{\large (}b^{p-3}-1{\large )}%
^{2}b^{3}-24\left( \allowbreak p(p-1)\right) ^{2}(p-2)(p-3){\large (}%
b^{p-3}-1{\large )(}b^{p-2}-1{\large )}b^{2}+$

$6\left( \allowbreak p(p-2)\right) ^{2}(p-1)(p-3){\large (}b^{p-3}-1{\large %
)(}b^{p-1}-1{\large )}b+18\left( \allowbreak p(p-1)(p-3)\right) ^{2}{\large (%
}b^{p-2}-1{\large )}^{2}b+$

$\left( \left( p-1\right) \left( p-2\right) \right) ^{2}p\left( p-3\right) 
{\large (}b^{p-3}-1{\large )}\left( b^{p}-1\right) -9\left( \allowbreak
p(p-3)\right) ^{2}(p-1)(p-2){\large (}b^{p-2}-1{\large )(}b^{p-1}-1{\large )}
$, which implies, after some simplification, that

\begin{eqnarray}
W(b) &=&12\left( p+1\right) b^{2p-3}+\left( p-2\right) \left( p+1\right)
\left( p^{2}-13p+6\right) b^{p}-  \notag \\
&&3p\left( p-3\right) \left( p^{2}-9p+2\right) b^{p-1}+3p\left( p-2\right)
\left( p-3\right) \left( p-5\right) b^{p-2}  \notag \\
&&-\left( p-1\right) (p-2)^{2}\left( p-3\right) b^{p-3}+8p(p-1)(p-2)^{2}b^{3}
\label{W} \\
&&-24p(p-1)(p-2)(p-3)b^{2}+6p\left( p-3\right) \left( 4p^{2}-16p+13\right) b
\notag \\
&&-2\left( p-2\right) \left( p-3\right) \left( 4p^{2}-12p-1\right) \text{.} 
\notag
\end{eqnarray}

Much of the work in proving Theorem \ref{T1} is embodied in the following
two propositions.

\begin{proposition}
\label{P5}Suppose that $p=n\in 
\mathbb{N}
,n\notin \left\{ 1,2,3\right\} $. Then $V(b)=Q(b)(b-1)^{5}$, where $Q$ is a
polynomial with negative nonzero coefficients.
\end{proposition}

\begin{proof}
While the cases $n=4$ thru $8$ could be absorbed into the proof below, we
find it more convenient to treat those cases separately. $n=4$ gives $%
V(b)=\allowbreak -60\left( b-1\right) ^{5},n=5$ gives $V(b)=-36\left(
13b^{2}+10b+2\right) \left( b-1\right) ^{5},n=6$ gives $V(b)=\allowbreak
-12\left( 142b^{4}+161b^{3}+105b^{2}+35b+7\right) \left( b-1\right)
^{5}\allowbreak ,n=7$ gives $V(b)=\allowbreak -24\left(
185b^{6}+246b^{5}+220b^{4}+140b^{3}+60b^{2}+20b+4\right) \allowbreak \left(
b-1\right) ^{5}$, and $n=8$ gives $V(b)=\allowbreak -36\left(
265b^{8}+385b^{7}+395b^{6}+327b^{5}+210b^{4}+105b^{3}+45b^{2}+15b+3\right)
\allowbreak \left( b-1\right) ^{5}$, so that Proposition \ref{P5} holds in
those cases. So assume now that $n\geq 9$. We list the derivatives of $V$
evaluated at $b=0$(simplified somewhat) and which are required for our
proof. 
\begin{eqnarray}
V(0) &=&\allowbreak 12n+12\text{, }V^{(i)}(0)=\allowbreak 0,i=1,...,n-4 
\notag \\
V^{(n-3)}(0) &=&(n-2)!\left( n+1\right) \left( n^{2}-13n+6\right)  \notag \\
V^{(n-2)}(0) &=&-3(n-2)!\allowbreak \left( n-3\right) n\left(
n^{2}-9n+2\right)  \label{DV} \\
V^{(n-1)}(0) &=&3\allowbreak n!\left( n-2\right) \left( n-3\right) \left(
n-5\right) \text{ }  \notag \\
V^{(n)}(0) &=&-n!\left( n-1\right) \left( n-2\right) ^{2}\left( n-3\right) 
\notag \\
V^{(i)}(0) &=&\allowbreak 0,i=n+1,...,2n-7\text{.}  \notag
\end{eqnarray}%
Note first that $Q(0)=-V(0)<0$, so we only need to show that $Q^{(r)}(0)\leq
0$ for $r\geq 1$. $Q(b)=\left( b-1\right) ^{-5}V(b)$ yields 
\begin{gather*}
Q^{(r)}(b)=\dfrac{d^{r}}{dy^{r}}{\large (}\left( b-1\right) ^{-5}V(b){\large %
)} \\
=\dsum\limits_{j=0}^{r}j!\dbinom{r}{j}\dbinom{-5}{j}\left( b-1\right)
^{-5-j}V^{(r-j)}(b)\text{,}
\end{gather*}

which implies that $Q^{(r)}(0)=\dsum\limits_{j=0}^{r}j!\dbinom{r}{j}\dbinom{%
-5}{j}\left( -1\right) ^{j+1}V^{(r-j)}(0)$. Using the identity $\dbinom{-5}{j%
}=\left( -1\right) ^{j}\dbinom{j+4}{j}$ yields $Q^{(r)}(0)=-\dsum%
\limits_{j=0}^{r}j!\dbinom{r}{j}\dbinom{j+4}{j}V^{(r-j)}(0)$ or 
\begin{equation}
Q^{(r)}(0)=-V^{(r)}(0)-\dsum\limits_{j=1}^{r}\left(
\tprod\limits_{i=0}^{j-1}(r-i)\right) \dbinom{j+4}{j}V^{(r-j)}(0)\text{.}
\label{jrn}
\end{equation}

\textbf{Case 1:} $1\leq r\leq n-4$. By (\ref{DV}), in (\ref{jrn}) $%
V^{(r)}(0)=0$ and $V^{(r-j)}(0)=0$ for $1\leq j\leq r-1$, so we are left
with $j=r$, which yields $Q^{(r)}(0)=-(r!)\dbinom{r+4}{r}V(0)=-12(r!)\dbinom{%
r+4}{r}(n+1)<0$

\textbf{Case 2:} $r=n-k,k=0,1,2,3$.

\qquad If $r=n-3$, then by (\ref{DV}) the only nonzero derivatives which
appear in (\ref{jrn}) are $-V^{(r)}(0)$ or when $j=r$, which gives

\begin{gather*}
Q^{(n-3)}(0)=-V^{(n-3)}(0)-(n-3)!\dbinom{n+1}{n-3}V(0)= \\
-(n-2)!\left( n+1\right) \left( n^{2}-13n+6\right) -12(n-3)!\dbinom{n+1}{n-3}%
(n+1)\text{,}
\end{gather*}

which implies that 
\begin{gather*}
-\dfrac{Q^{(n-3)}(0)}{\left( n+1\right) (n-3)!}=(n-2)\left(
n^{2}-13n+6\right) +\dfrac{1}{2}(n+1)n(n-1)(n-2) \\
=\dfrac{1}{2}\left( n-2\right) \left( n-4\right) \left( n^{2}+6n-3\right) >0%
\text{.}
\end{gather*}

$\allowbreak $Since $n^{2}+6n-3>0,n>1$, $Q^{(n-3)}(0)<0$.

\qquad If $r=n-2$, then by (\ref{DV}) the only nonzero derivatives which
appear in (\ref{jrn}) are $-V^{(r)}(0)$ or when $j=1$ or $j=r$ in (\ref{jrn}%
), which gives

\begin{eqnarray*}
Q^{(n-2)}(0) &=&-V^{(n-2)}(0)-5(n-2)V^{(n-3)}(0)-(n-2)!\dbinom{n+2}{n-2}V(0)=
\\
&&3(n-2)!\allowbreak \left( n-3\right) n\left( n^{2}-9n+2\right) - \\
&&5(n-2)(n-2)!\left( n+1\right) \left( n^{2}-13n+6\right) \\
&&-12(n-2)!\dbinom{n+2}{n-2}(n+1)\text{,}
\end{eqnarray*}

which implies that%
\begin{gather*}
\dfrac{Q^{(n-2)}(0)}{(n-2)!}=3\left( n-3\right) n\left( n^{2}-9n+2\right) \\
-5(n-2)\left( n+1\right) \left( n^{2}-13n+6\right) \\
-\dfrac{1}{2}(n+2)(n+1)^{2}n(n-1) \\
=-\dfrac{1}{2}\left( n-5\right) \left( n+2\right) p(n)\text{,}
\end{gather*}

where $p(x)=x^{3}+10x^{2}-27x+12$. $p^{\prime }(x)=\allowbreak
3x^{2}+20x-27>0$ for $x>2$, which implies that $p$ is increasing on $%
(2,\infty )$. Since $p(2)>0,p(x)>0$ for $x>2$. Thus $Q^{(n-2)}(0)<0$.

\qquad If $r=n-1$, then by (\ref{DV}) the only nonzero derivatives which
appear in (\ref{jrn}) are $-V^{(r)}(0)$ or when $j=1,j=2$, or $j=r$, which
gives

\begin{eqnarray*}
Q^{(n-1)}(0) &=&-V^{(n-1)}(0)-5(n-1)V^{(n-2)}(0) \\
&&-15(n-1)(n-2)V^{(n-3)}(0) \\
&&-(n-1)!\dbinom{n+3}{n-1}V(0) \\
&=&-3\allowbreak n!\left( n-2\right) \left( n-3\right) \left( n-5\right) + \\
&&15(n-1)(n-2)!\allowbreak \left( n-3\right) n\left( n^{2}-9n+2\right) \\
&&-15(n-1)(n-2)(n-2)!\left( n+1\right) \left( n^{2}-13n+6\right) \\
&&-12(n-1)!\dbinom{n+3}{n-1}(n+1)\text{,}
\end{eqnarray*}

which implies that%
\begin{eqnarray*}
\dfrac{Q^{(n-1)}(0)}{(n-1)!} &=&-3\allowbreak n\left( n-2\right) \left(
n-3\right) \left( n-5\right) + \\
&&15\allowbreak \left( n-3\right) n\left( n^{2}-9n+2\right) \\
&&-15(n-2)\left( n+1\right) \left( n^{2}-13n+6\right) \\
&&-\dfrac{1}{2}(n+3)(n+2)(n+1)^{2}n \\
&=&-\dfrac{1}{2}\left( n-1\right) \left( n-6\right) \left( n+3\right) \left(
n^{2}+17n-20\right) \text{.}
\end{eqnarray*}

Since $n^{2}+17n-20>0,n>2$, $Q^{(n-1)}(0)<0$.$\allowbreak $

\qquad If $r=n$, then by (\ref{DV}) the only nonzero derivatives which
appear in (\ref{jrn}) are $-V^{(r)}(0)$ or when $j=1,j=2,j=3$, or $j=r$,
which gives

\begin{gather*}
Q^{(n)}(0)=-V^{(n)}(0)-5nV^{(n-1)}(0)-15n(n-1)V^{(n-2)}(0) \\
-35n(n-1)(n-2)V^{(n-3)}(0)-n!\dbinom{n+4}{n}V(0)= \\
n!\left( n-1\right) \left( n-2\right) ^{2}\left( n-3\right) -15n\left(
n!\right) \left( n-2\right) \left( n-3\right) \left( n-5\right) + \\
45n(n-1)(n-2)!\allowbreak \left( n-3\right) n\left( n^{2}-9n+2\right) \\
-35n(n-1)(n-2)(n-2)!\left( n+1\right) \left( n^{2}-13n+6\right) -12n!\dbinom{%
n+4}{n}(n+1)\text{,}
\end{gather*}

which implies that 
\begin{gather*}
\dfrac{Q^{(n)}(0)}{n!}=\left( n-1\right) \left( n-2\right) ^{2}\left(
n-3\right) -15n\left( n-2\right) \left( n-3\right) \left( n-5\right) + \\
45\allowbreak \left( n-3\right) n\left( n^{2}-9n+2\right) -35(n-2)\left(
n+1\right) \left( n^{2}-13n+6\right) \\
-\dfrac{1}{2}(n+4)(n+3)(n+2)(n+1)^{2}\allowbreak = \\
-\dfrac{1}{2}\left( n+4\right) \left( n-7\right) p(n)\text{,}
\end{gather*}

where $p(x)=x^{3}+22x^{2}-45x+30$. $p^{\prime }(x)=\allowbreak
3x^{2}+44x-45>0$ for $x>1$, which implies that $p$ is increasing on $%
(1,\infty )$. Since $p(1)>0,p(x)>0$ for $x>1$. Thus $Q^{(n)}(0)<0$.

Note that we only need to go up to $r=2n-8$ since $\deg Q=2n-8$. Since $%
n>8,2n-8>n$.

\textbf{Case 3:} $r=n+k,k=1,...,n-8$. Note that, by (\ref{DV}), in (\ref{jrn}%
) the only other nonzero derivatives(not including the $0$th derivative)
which appear are when

$r-j=n-l,l=0,1,2,3\Rightarrow j=r-n+l=k+l$ for $l=0,1,2,3$. So let $%
j=k,k+1,k+2,k+3$ and also let $j=r$ in (\ref{jrn}) to obtain

\begin{gather*}
Q^{(n+k)}(0)=-\left( \prod\limits_{i=0}^{k-1}(n+k-i)\right) \dbinom{k+4}{k}%
V^{(n)}(0)- \\
\left( \prod\limits_{i=0}^{k}(n+k-i)\right) \dbinom{k+5}{k+1}V^{(n-1)}(0) \\
-\left( \prod\limits_{i=0}^{k+1}(n+k-i)\right) \dbinom{k+6}{k+2}%
V^{(n-2)}(0)-\left( \prod\limits_{i=0}^{k+2}(n+k-i)\right) \dbinom{k+7}{k+3}%
V^{(n-3)}(0) \\
-(n+k)!\dbinom{n+k+4}{n+k}V(0)= \\
\left( \prod\limits_{i=0}^{k-1}(n+k-i)\right) n!\left( n-1\right) \left(
n-2\right) ^{2}\left( n-3\right) \dfrac{1}{24}\left(
\prod\limits_{i=1}^{4}(k+i)\right) \\
-3\left( \prod\limits_{i=0}^{k}(n+k-i)\right) n!\left( n-2\right) \left(
n-3\right) \left( n-5\right) \dfrac{1}{24}\left(
\prod\limits_{i=2}^{5}(k+i)\right) \\
+3\left( \prod\limits_{i=0}^{k+1}(n+k-i)\right) (n-2)!\left( n-3\right)
n\left( n^{2}-9n+2\right) \dfrac{1}{24}\left(
\prod\limits_{i=3}^{6}(k+i)\right) \\
-\left( \prod\limits_{i=0}^{k+2}(n+k-i)\right) (n-2)!\left( n+1\right)
\left( n^{2}-13n+6\right) \dfrac{1}{24}\left(
\prod\limits_{i=4}^{7}(k+i)\right) \\
-\dfrac{1}{2}(n+k)!\left( \prod\limits_{i=1}^{4}(n+k+i)\right) (n+1)\text{,}
\end{gather*}

which implies that%
\begin{gather*}
\dfrac{24Q^{(n+k)}(0)}{\left( \prod\limits_{i=0}^{k-1}(n+k-i)\right) }%
=n!\left( n-1\right) \left( n-2\right) ^{2}\left( n-3\right) \left(
\prod\limits_{i=1}^{4}(k+i)\right) \\
-3n(n-1)!n\left( n-2\right) \left( n-3\right) \left( n-5\right) \left(
\prod\limits_{i=2}^{5}(k+i)\right) \\
+3n(n-1)(n-2)!\left( n-3\right) n\left( n^{2}-9n+2\right) \left(
\prod\limits_{i=3}^{6}(k+i)\right) \\
-n(n-1)(n-2)(n-3)!(n-2)\left( n+1\right) \left( n^{2}-13n+6\right) \left(
\prod\limits_{i=4}^{7}(k+i)\right) \\
-12(n+1)!\left( \prod\limits_{i=1}^{4}(n+k+i)\right) \text{,}
\end{gather*}

therefore%
\begin{gather*}
\dfrac{24Q^{(n+k)}(0)}{(n-3)!\left( \prod\limits_{i=0}^{k-1}(n+k-i)\right) }%
=n\left( n-1\right) ^{2}\left( n-2\right) ^{3}\left( n-3\right) \left(
\prod\limits_{i=1}^{4}(k+i)\right) \\
-3n^{2}(n-1)\left( n-2\right) ^{2}\left( n-3\right) \left( n-5\right) \left(
\prod\limits_{i=2}^{5}(k+i)\right) \\
+3(n-2)\left( n-3\right) n^{2}(n-1)\left( n^{2}-9n+2\right) \left(
\prod\limits_{i=3}^{6}(k+i)\right) \\
-(n-2)^{2}\left( n+1\right) n(n-1)\left( n^{2}-13n+6\right) \left(
\prod\limits_{i=4}^{7}(k+i)\right) \\
-12(n+1)n(n-1)(n-2)\left( \prod\limits_{i=1}^{4}(n+k+i)\right) \text{,}
\end{gather*}

thus%
\begin{gather*}
\dfrac{24Q^{(n+k)}(0)}{(n!)\left( \prod\limits_{i=0}^{k-1}(n+k-i)\right) }%
=\left( n-1\right) \left( n-2\right) ^{2}\left( n-3\right) \left(
\prod\limits_{i=1}^{4}(k+i)\right) \\
-3n\left( n-2\right) \left( n-3\right) \left( n-5\right) \left(
\prod\limits_{i=2}^{5}(k+i)\right) + \\
3\left( n-3\right) n\left( n^{2}-9n+2\right) \left(
\prod\limits_{i=3}^{6}(k+i)\right) \\
-(n-2)\left( n+1\right) \left( n^{2}-13n+6\right) \left(
\prod\limits_{i=4}^{7}(k+i)\right) -12(n+1)\left(
\prod\limits_{i=1}^{4}(n+k+i)\right) \allowbreak =
\end{gather*}

\begin{gather*}
-12\left( n-7-k\right) \left( n+k+4\right) {\large (}\allowbreak
n^{3}+2\left( 3k+11\right) n^{2}+ \\
\left( k^{2}-7k-45\right) n+\left( k+5\right) \allowbreak \left( k+6\right) 
{\large )}=\allowbreak \\
12\left( n-7-k\right) \left( n+k+4\right) p(n)\text{,}
\end{gather*}%
where 
\begin{equation*}
p(x)=x^{3}+2\left( 3k+11\right) x^{2}+\left( k^{2}-7k-45\right) x+\left(
k+5\right) \allowbreak \left( k+6\right) \text{.}
\end{equation*}%
Since $k^{2}-7k-45>0$ f$\allowbreak $or $k\geq 12,p$ has all positive
coefficients f$\allowbreak $or $k\geq 12$, and hence no roots in $%
\mathbb{N}
$. We now show that $p$ has only one real root when $0\leq k\leq 11$. $%
\lim\limits_{x\rightarrow -\infty }p(x)=\allowbreak -\infty $ and $p(0)>0$
implies that $p$ has a negative real root. In general, the polynomial $%
y=x^{3}+A_{1}x^{2}+A_{2}x+A_{3}$ has all real roots if and only if its
discriminant, $%
D=18A_{1}A_{2}A_{3}+A_{1}^{2}A_{2}^{2}-27A_{3}^{2}-4A_{2}^{3}-4A_{1}^{3}A_{3} 
$, is non--negative. The discriminant of $p$, after simplifying, is the
polynomial in $k$ given by $D(k)=32k^{6}-912k^{5}-22\,943k^{4}-175%
\,366k^{3}-606\,963\allowbreak k^{2}-947\,132k-492\,060$. $D$ has one
positive real root by Descarte rule of signs since there is one sign change
in $D$. Since $D(46)<0$ and $D(47)>0,D(k)=0$ for $46<k<47$. Since $D(0)<0$
and $D$ cannot vanish in $[0,11]$, $D(k)<0$ for $0\leq k\leq 11$. Hence $p$
has only one real root when $0\leq k\leq 11$. Since that real root is
negative and $p(0)>0$, $p(n)>0$ for $n\in 
\mathbb{N}
$ and $0\leq k\leq 11$. That proves that $\dfrac{24Q^{(n+k)}(0)}{(n!)\left(
\prod\limits_{i=0}^{k-1}(n+k-i)\right) }=-12\left( n-7-k\right) \left(
n+k+4\right) p(n)<0$ since $k<n-7$, which in turn implies that $%
Q^{(n+k)}(0)<0$.
\end{proof}

\begin{proposition}
\label{P6}Suppose that $p$ is a negative integer. Then $V\left( \dfrac{1}{b}%
\right) =S(b)(b-1)^{5}$, where $S$ is a polynomial with positive nonzero
coefficients.
\end{proposition}

\begin{proof}
Let $p=-n,n\in 
\mathbb{N}
$ and let $K(b)=V\left( \dfrac{1}{b}\right) ,b>0$. Then by (\ref{V}), 
\begin{eqnarray*}
K(b) &=&-12(n-1)+8n(n+1)(n+2)^{2}b^{2n+6} \\
&&-24n(n+1)(n+2)(n+3)b^{2n+5} \\
&&+6n(n+3)\left( 4n^{2}+16n+13\right) b^{2n+4} \\
&&-2(n+2)(n+3)\left( 4n^{2}+12n-1\right) b^{2n+3} \\
&&+(n+2)(n-1)\left( n^{2}+13n+6\right) b^{n+3}-3(n+3)n\left(
n^{2}+9n+2\right) b^{n+2}+ \\
&&3n(n+2)(n+3)\left( n+5\right) b^{n+1}-(n+1)(n+2)^{2}(n+3)b^{n}\text{,}
\end{eqnarray*}%
and $S(b)=\dfrac{K(b)}{\left( b-1\right) ^{5}}$. As in the proof of
Proposition \ref{P5}, with $V(b)$ replaced by $K(b)$, we have $%
S^{(r)}(0)=-K^{(r)}(0)-\dsum\limits_{j=0}^{r}j!\dbinom{r}{j}\dbinom{j+4}{j}%
K^{(r-j)}(0)\ $or 
\begin{equation}
S^{(r)}(0)=-K^{(r)}(0)-\dsum\limits_{j=1}^{r}\left(
\prod\limits_{i=0}^{j-1}(r-i)\right) \dbinom{j+4}{j}K^{(r-j)}(0)\text{.}
\label{jrneg}
\end{equation}

It is more convenient to do $n=1$\textbf{\ }separately. In that case, $%
V\left( \dfrac{1}{b}\right) =\allowbreak 72b\left( 2b^{2}+2b+1\right) \left(
b-1\right) ^{5}$ and Proposition \ref{P6} holds. So assume now that $n\geq 2$%
.

We list the derivatives of $K$ evaluated at $b=0$(simplified somewhat) and
which are required for our proof. 
\begin{eqnarray}
K(0) &=&\allowbreak -12(n-1)\text{; }K^{(i)}(0)=\allowbreak 0,i=1,...,n-1 
\notag \\
K^{(n)}(0) &=&-(n+3)!(n+2)\text{; }K^{(n+1)}(0)=3(n+3)!n\left( n+5\right) 
\notag \\
K^{(n+2)}(0) &=&-3(n+3)!n\left( n^{2}+9n+2\right)  \label{DK} \\
K^{(n+3)}(0) &=&(n+3)!(n+2)(n-1)\left( n^{2}+13n+6\right)  \notag \\
K^{(i)}(0) &=&\allowbreak 0,i=n+4,...,2n+2\text{.}  \notag
\end{eqnarray}

Note first that $S(0)=-K(0)>0$, so we only need to show that $S^{(r)}(0)\geq
0$ for $r\geq 1$.

\textbf{Case 1:} $1\leq r\leq n-1$

By (\ref{DK}), in (\ref{jrneg}), $K^{(r)}(0)=0$ and $K^{(r-j)}(0)=0$ for $%
1\leq j\leq r-1$, so we are left with $S^{(r)}(0)=-(r!)\dbinom{r+4}{r}%
K(0)=12(r!)\dbinom{r+4}{r}(n-1)>0$

\textbf{Case 2:} $r=n+k,k=0,1,2$

\qquad If $r=n$, then by (\ref{DK}) the only nonzero derivatives which
appear in (\ref{jrneg}) are $-K^{(n)}(0)$ or when $j=r$, which gives

\begin{gather*}
S^{(n)}(0)=-K^{(n)}(0)-(n!)\dbinom{n+4}{n}K(0)= \\
(n+3)!(n+2)+(n!)\dbinom{n+4}{n}12(n-1)>0\text{.}
\end{gather*}

\qquad If $r=n+1$, then by (\ref{DK}) the only nonzero derivatives which
appear in (\ref{jrneg}) are $-K^{(n+1)}(0)$, or when $j=1$, or $j=r$, which
gives

\begin{gather*}
S^{(n+1)}(0)=-K^{(n+1)}(0)-5(n+1)K^{(n)}(0)-(n+1)!\dbinom{n+5}{n+1}K(0) \\
=-3(n+3)!n\left( n+5\right) +5(n+3)!(n+2)(n+1)+12(n-1)(n+1)!\dbinom{n+5}{n+1}%
\text{,}
\end{gather*}

and thus 
\begin{eqnarray*}
\dfrac{S^{(n+1)}(0)}{(n+1)!} &=&-3n(n+2)(n+3)\left( n+5\right)
+5(n+1)(n+2)^{2}(n+3) \\
&&+\dfrac{1}{2}(n-1)\left( n+5\right) (n+4)(n+3)(n+2) \\
&=&\allowbreak \dfrac{1}{2}\left( n+11\right) \left( n+3\right) \left(
n+2\right) \left( n+1\right) n>0\text{.}
\end{eqnarray*}

\qquad If $r=n+2$, then by (\ref{DK}) the only nonzero derivatives which
appear in (\ref{jrneg}) are $-K^{(n+2)}(0)$, or when $j=1$, $j=2$, or $j=r$,
which gives

\begin{gather*}
S^{(n+2)}(0)=-K^{(n+2)}(0)-5(n+2)K^{(n+1)}(0)-15(n+2)(n+1)K^{(n)}(0) \\
-(n+2)!\dbinom{n+6}{n+2}K(0)=3(n+3)!n\left( n^{2}+9n+2\right) \\
-15(n+2)(n+3)!n\left( n+5\right) +15(n+3)!(n+2)^{2}(n+1)+ \\
12(n-1)(n+2)!\dbinom{n+6}{n+2}\text{,}
\end{gather*}

which implies that%
\begin{gather*}
\dfrac{S^{(n+2)}(0)}{3(n+3)!}=n\left( n^{2}+9n+2\right) -5n(n+2)\left(
n+5\right) +5(n+1)(n+2)^{2} \\
+\dfrac{1}{6}(n-1)(n+6)\left( n+5\right) (n+4)=\allowbreak \dfrac{1}{6}np(n)%
\text{,}
\end{gather*}

where $p(x)=x^{3}+20x^{2}+53x-2$. $p$ has one positive real root by Descarte
rule of signs. Since $p(0)<0$ and $p(1)>0$, that root lies in $(0,1)$. Hence 
$p(n)>0$ for $n\geq 1$, which implies that $S^{(n+2)}(0)>0$.

\qquad If $r=n+3$, then by (\ref{DK}) the only nonzero derivatives which
appear in (\ref{jrneg}) are $-K^{(n+3)}(0)$, or when $j=1$, $j=2,j=3$, or $%
j=r$, which gives

\begin{gather*}
S^{(n+3)}(0)=-K^{(n+3)}(0)-5(n+3)K^{(n+2)}(0)-15(n+3)(n+2)K^{(n+1)}(0) \\
-35(n+3)(n+2)(n+1)K^{(n)}(0)-(n+3)!\dbinom{n+7}{n+3}K(0)= \\
-(n+3)!(n+2)(n-1)\left( n^{2}+13n+6\right) +15(n+3)(n+3)!n\left(
n^{2}+9n+2\right) \\
-45(n+3)(n+2)(n+3)!n\left( n+5\right) +35(n+3)(n+2)(n+1)(n+3)!(n+2) \\
+12\dbinom{n+7}{n+3}(n-1)\text{,}
\end{gather*}

which implies that 
\begin{gather*}
\dfrac{S^{(n+3)}(0)}{(n+3)!}=-(n+2)(n-1)\left( n^{2}+13n+6\right)
+15(n+3)n\left( n^{2}+9n+2\right) \\
-45(n+3)(n+2)n\left( n+5\right) +35(n+3)(n+2)(n+1)(n+2)+ \\
\dfrac{1}{2}(n+7)(n+6)(n+5)(n+4)(n-1)= \\
\dfrac{1}{2}\left( n-1\right) \left( n+4\right) p(n)\text{,}
\end{gather*}%
$\allowbreak $

where $p(x)=x^{3}+26x^{2}+75x-6$. $p$ has one positive real root by Descarte
rule of signs. Since $p(0)<0$ and $p(1)>0$, that root lies in $(0,1)$. Hence 
$p(n)>0$ for $n\geq 1$, which implies that $S^{(n+3)}(0)>0$.$\allowbreak
\allowbreak \allowbreak $

\qquad Note that we only need to go up to $r=2n+1$ since $\deg S=2n+1$. So
consider

\textbf{Case 3:} $r=n+k,k=4,...,n+1$

Note that, by (\ref{DK}), in (\ref{jrneg}) the only nonzero derivatives
which appear in (\ref{jrneg})(not including the $0$th derivative) are when

$r-j=n+l,l=0,1,2,3$, or when $r-j=0$. That gives $j=r-n-l=k-l$ for $%
l=0,1,2,3 $ or $j=r$. So let $j=k,k-1,k-2,k-3$ and also let $j=r=n+k$ in (%
\ref{jrneg}) to obtain%
\begin{gather*}
S^{(n+k)}(0)=-\left( \prod\limits_{i=0}^{k-1}(n+k-i)\right) \dbinom{k+4}{k}%
K^{(n)}(0) \\
-\left( \prod\limits_{i=0}^{k-2}(n+k-i)\right) \dbinom{k+3}{k-1}%
K^{(n+1)}(0)-\left( \prod\limits_{i=0}^{k-3}(n+k-i)\right) \dbinom{k+2}{k-2}%
K^{(n+2)}(0) \\
-\left( \prod\limits_{i=0}^{k-4}(n+k-i)\right) \dbinom{k+1}{k-3}%
K^{(n+3)}(0)-(n+k)!\dbinom{n+k+4}{n+k}K(0)=
\end{gather*}

\begin{gather*}
\left( \prod\limits_{i=0}^{k-1}(n+k-i)\right) \dbinom{k+4}{k}(n+3)!(n+2) \\
-3\left( \prod\limits_{i=0}^{k-2}(n+k-i)\right) \dbinom{k+3}{k-1}%
(n+3)!n\left( n+5\right) + \\
3\left( \prod\limits_{i=0}^{k-3}(n+k-i)\right) \dbinom{k+2}{k-2}%
(n+3)!n\left( n^{2}+9n+2\right) \\
-\left( \prod\limits_{i=0}^{k-4}(n+k-i)\right) \dbinom{k+1}{k-3}%
(n+3)!(n+2)(n-1)\left( n^{2}+13n+6\right) + \\
12(n-1)(n+k)!\dbinom{n+k+4}{n+k}\text{,}
\end{gather*}%
which implies that%
\begin{gather*}
\dfrac{24S^{(n+k)}(0)}{\left( \prod\limits_{i=0}^{k-4}(n+k-i)\right) (n+3)!}%
=\left( \prod\limits_{i=k-3}^{k-1}(n+k-i)\right) \left(
\prod\limits_{i=1}^{4}(k+i)\right) (n+2) \\
-3\left( \prod\limits_{i=k-3}^{k-2}(n+k-i)\right) \left(
\prod\limits_{i=1}^{4}(k-1+i)\right) n\left( n+5\right) + \\
3\left( n+3\right) \left( \prod\limits_{i=1}^{4}(k-2+i)\right) n\left(
n^{2}+9n+2\right) \\
-\left( \prod\limits_{i=1}^{4}(k-3+i)\right) (n+2)(n-1)\left(
n^{2}+13n+6\right) +12(n-1)\left( \prod\limits_{i=1}^{4}(n+k+i)\right) \\
=12\left( n+2-k\right) \left( n+k+1\right) {\large (}n^{3}+\left(
6k+8\right) n^{2}+\left( k^{2}+17k+15\right) n+k-k^{2}{\large )}\text{.}
\end{gather*}%
Let $q(x)=x^{3}+\left( 6k+8\right) x^{2}+\left( k^{2}+17k+15\right)
x+k-k^{2} $. Since $k\geq 3$, $q$ has one sign change, so one positive real
root by Descarte rule of signs. Since $q(0)=k-k^{2}<0$ and $q(1)=\allowbreak
24+24k>0 $, that positive real root lies between $0$ and $1$. Thus $q(n)>0$
for $n\geq 1$ and that proves that $S^{(n+k)}(0)>0$ for $k\leq n+1$.
\end{proof}

\begin{proof}
(of Theorem \ref{T1})Let $z_{1}=x_{1}+iy_{1}$. Then we want to prove that $%
\dint\limits_{a}^{b}f\,^{\prime \prime \prime \prime }(t)(z_{1}-t)^{3}dt=0$
with $a<\func{Re}z_{1}<b$. Theorem \ref{T1} will then follow from (\ref{1}).
Given $0<a<b$, suppose that $\dint\limits_{1}^{b/a}f\,^{\prime \prime \prime
\prime }(t)(z_{1}-t)^{3}dt=\allowbreak 0$ where $1<z_{1}<\dfrac{b}{a}$. Then 
$\dint\limits_{1}^{b/a}t^{s}(z_{1}-t)^{3}dt=\allowbreak 0$, where $%
s=p(p-1)(p-2)(p-3)$. Letting $u=at$, we have $0=\dint%
\limits_{1}^{b/a}t^{s}(z_{1}-t)^{3}dt=\dfrac{1}{a^{s+1}}\dint%
\limits_{a}^{b}u^{s}\left( z_{1}-\dfrac{u}{a}\right) ^{3}dt=\dfrac{1}{a^{s+4}%
}\dint\limits_{a}^{b}u^{s}(az_{1}-u)^{3}du$ with $a<az_{1}<b$. That shows
that it suffices to prove Theorem \ref{T1} when $a=1$, which we assume for
the rest of the proof.

Suppose first that $p=n$, a positive integer, $n\notin \left\{ 1,2,3\right\} 
$. By Proposition \ref{P5}, $V(b)=Q(b)(b-1)^{5}$, where $Q$ is a polynomial
with negative nonzero coefficients. Thus $V(b)>0$ for $0<b<1$ and $V(b)<0$
for $b>1$. Now suppose that $p=-n,$ $n\in 
\mathbb{N}
$. By Proposition \ref{P6}, $V\left( \dfrac{1}{b}\right) =S(b)(b-1)^{5}$,
where $S$ is a polynomial with positive nonzero coefficients. It follows
again that $V(b)>0$ for $0<b<1$ and $V(b)<0$ for $b>1$. Assuming now that $p$
is an integer, $p\notin \left\{ 0,1,2,3\right\} $, by (\ref{V}) and (\ref{W}%
), 
\begin{gather*}
b^{2p-3}V\left( \dfrac{1}{b}\right) =12(p+1)b^{2p-3}-2(p-2)(p-3)\left(
4p^{2}-12p-1\right) \\
-2(p-2)(p-3)\left( 4p^{2}-12p-1\right) + \\
6p(p-3)\left( 4p^{2}-16p+13\right) b-24p(p-1)(p-2)(p-3)b^{2}+ \\
8p(p-1)(p-2)^{2}b^{3}-\left( p-1\right)
(p-2)^{2}(p-3)b^{p-3}+3p(p-2)(p-3)\left( p-5\right) b^{p-2} \\
-3(p-3)p\left( p^{2}-9p+2\right) b^{p-1}+(p-2)\left( p+1\right) \left(
p^{2}-13p+6\right) b^{p}=W(b)\text{.}
\end{gather*}

That is, 
\begin{equation}
W(b)=b^{2p-3}V\left( \dfrac{1}{b}\right) ,b>0\text{.}  \label{VWR}
\end{equation}

Using (\ref{VWR}), it then follows immediately that $W(b)<0$ for $0<b<1$ and 
$W(b)>0$ for $b>1$. Let $g(x)$ be given by (\ref{g}), where $A,B,C$, and $D$
are given by (\ref{ABCDp}). Since $p(p-1)>0,V(b)<0$ and $W(b)>0$ for $b>1$
implies, using (\ref{VW}), that $g(1)<0$ and $g(b)>0$ for $b>1$. By the
Intermediate Value Theorem, $g(x_{1})=0$ for some $a<x_{1}<b$. $f(z)=z^{p}$
clearly satisfies the hypotheses of Proposition \ref{P3}, which then implies
that $P_{b}(z_{1})=P_{a}(z_{1})$, where $y_{1}\neq 0$ is given by
Proposition \ref{P3} and $z_{1}=x_{1}+iy_{1}$.
\end{proof}

\textbf{Remark: }Theorem \ref{T1} probably holds in the more general case
when $p\in \Re ,p>3$ or $p<0$. If $p=\allowbreak \dfrac{n}{m}$ is rational,
then a proof similar to the proofs of Proposition \ref{P5} or Proposition %
\ref{P6} might work to obtain a similar factorization of $V\left(
b^{1/m}\right) $. After trying some of the details, it looks somewhat
tedious, and a different approach might lead to proving Theorem \ref{T1} for
a much larger class of functions, such as $f(z)=e^{z}$.

\section{Alternate Proofs and Partial Results\label{S3}}

In this section we give for $f(z)=z^{p}$ and $r=3$ which are not covered by
Theorem \ref{T1}. The results are partial because we either prove that $%
a<x_{1}$ for certain real values of $p$ or that $x_{1}<b$ for certain real
values of $p$, but not both. We also give an alternate proof of\textbf{\ }%
Theorem \ref{T1} when $p\in N,p\geq 13$ which is somewhat different than the
proof of Theorem \ref{T1} given above. First we need the following lemma.

\begin{lemma}
\label{L4}

(i) Let $k(x)=3\dfrac{r-1}{r}\dfrac{x^{r}-1}{x^{r-1}-1}-x$, where $r\geq 
\dfrac{3}{2}$. Then $k(x)>2$ for $x>1$

(ii) Let $l(x)=3\dfrac{r-1}{r}\dfrac{x^{r}-1}{x^{r-1}-1}-2x-1$, where $r<0$.
Then $l(x)<0$ for $x>1$
\end{lemma}

\begin{proof}
Consider the family of means

$E_{r,s}(x,y)=\left\{ 
\begin{array}{ll}
\left( \dfrac{s}{r}\dfrac{x^{r}-y^{r}}{x^{s}-y^{s}}\right) ^{1/(r-s)} & 
\text{if }r,s\neq 0,r\neq s,x\neq y \\ 
\left( \dfrac{1}{r}\dfrac{x^{r}-y^{r}}{\log x-\log y}\right) ^{1/r} & \text{%
if }r\neq 0,s=0,x\neq y \\ 
e^{-1/r}\left( \dfrac{x^{x^{r}}}{y^{y^{r}}}\right) ^{1/(x^{r}-y^{r})} & 
\text{if }s=r\neq 0,x\neq y \\ 
\sqrt{xy} & \text{if }r=s=0,x\neq y \\ 
x & \text{if }x=y%
\end{array}%
\right. $, known as the Stolarsky means. It is well known(\cite{LS}) that,
for fixed $x$ and $y$, $E_{r,s}(x,y)$ is increasing in the parameters $r$
and $s$.

To prove (i), $k(x)=3E_{r-1,r}(x,1)-x>3E_{1/2,3/2}(x,1)-x=\sqrt{x}+1>2$

To prove (ii), $l(x)=3E_{r-1,r}(x,1)-2x-1<3E_{-1,0}(x,1)-2x-1=\dfrac{x\ln x}{%
x-1}-2x-1$. From $\ln x<x-1$ we have $\dfrac{\ln x}{x-1}<1<2+\dfrac{1}{x}$
for $x>1$, which implies that $\dfrac{x\ln x}{x-1}<2x+1$.
\end{proof}

We now use Lemma \ref{L4} to prove part of the conclusion of Theorem \ref{T1}
for $p\in \Re $, $p\geq \dfrac{7}{2}$ or for $p\in \Re $, $p<2,p\neq 0,1$.

\begin{theorem}
\label{T2}Suppose that $f(z)=z^{p}$, $p\in \Re $, and let $P_{c}$ denote the
Taylor polynomial to $f$ of order $3$ at $x=c$.

(i) If $p\geq \dfrac{7}{2}$, then for any $0<a<b$, $P_{b}-P_{a}$ has nonreal
roots $x_{1}\pm iy_{1}$, with $a<x_{1}$.

(ii) If $p<2,p\neq 0,1$, then for any $0<a<b$, $P_{b}-P_{a}$ has nonreal
roots $x_{1}\pm iy_{1}$, with $x_{1}<b$.
\end{theorem}

\begin{proof}
As in the proof of Theorem \ref{T1}, we may assume that $a=1$, so that $b>1$%
.\ Let 
\begin{eqnarray*}
Q(z) &=&\dfrac{6}{f\,^{\prime \prime \prime }(b)-f\,^{\prime \prime \prime
}(1)}{\large (}P_{b}(z)-P_{1}(z){\large )}= \\
&&z^{3}+a_{1}z^{2}+a_{2}z+a_{3}\text{,}
\end{eqnarray*}%
where $a_{1}=\dfrac{3{\large (}f\,^{\prime \prime }(b)-bf\,^{\prime \prime
\prime }(b)-f\,^{\prime \prime }(1)+f\,^{\prime \prime \prime }(1){\large )}%
}{f\,^{\prime \prime \prime }(b)-f\,^{\prime \prime \prime }(1)}=\allowbreak
-3\dfrac{p-3}{p-2}\dfrac{b^{p-2}-1}{b^{p-3}-1}$. Note that $Q$ and $%
P_{b}-P_{1}$ have the same roots. Write $Q(z)=(z-x_{0})(z-z_{1})(z-\bar{z}%
_{1})$, where $x_{0}$ is the real root of $Q(z)=0$ with $1<x_{0}<b$
guaranteed by Proposition \ref{P1}(i) with $r=3$. Then $x_{0}+2\func{Re}%
z_{1}=-a_{1}=\allowbreak 3\dfrac{p-3}{p-2}\dfrac{b^{p-2}-1}{b^{p-3}-1}$,
which implies that $\func{Re}z_{1}=\dfrac{1}{2}\left( 3\dfrac{s-1}{s}\dfrac{%
b^{s}-1}{b^{s-1}-1}-x_{0}\right) $, where $s=p-2$.

To prove (i), since $r\geq \dfrac{3}{2}$ and $x_{0}<b,\func{Re}z_{1}\geq 
\dfrac{1}{2}\left( 3\dfrac{s-1}{s}\dfrac{b^{s}-1}{b^{s-1}-1}-b\right) >1$ by
Lemma \ref{L4}(i). Thus we have shown that $P_{b}(z_{1})=P_{1}(z_{1})$ with $%
\func{Re}z_{1}>1$. To prove (ii), since $r<0$ and $x_{0}>1$, $\func{Re}%
z_{1}\leq \dfrac{1}{2}\left( 3\dfrac{s-1}{s}\dfrac{b^{s}-1}{b^{s-1}-1}%
-1\right) <b$ by Lemma \ref{L4}(ii). Thus we have shown that $%
P_{b}(z_{1})=P_{1}(z_{1})$ with $\func{Re}z_{1}<b$.
\end{proof}

\qquad We shall now give an alternate proof of Theorem \ref{T1} when $p\in
N,p\geq 13$(the cases $p\in N,p=4,...,12$ can be checked directly). The
method used here is somewhat different from the proof of Theorem \ref{T1}
and could possibly lead to a proof for $p>3$ in general. First we need the
following lemmas.

\begin{lemma}
\label{L5}For any $n\in 
\mathbb{N}
,n\geq 4$, and $j\leq n-4$%
\begin{gather*}
\dsum\limits_{k=j}^{n-4}(\allowbreak 8k^{3}+60k^{2}+130k+75)\dbinom{n}{k+4}%
(-1)^{k-j}\dbinom{k}{j}= \\
-\dfrac{1}{2}\left( n+j+1\right) {\large (}n^{2}-(10j+13)n+j^{2}+5j+6{\large %
)}\text{.}
\end{gather*}
\end{lemma}

\begin{proof}
One can first derive formulas for $\dsum\limits_{k=j}^{n-4}k^{i}\dbinom{n}{%
k+4}(-1)^{k-j}\dbinom{k}{j}$ for $i=0,1,2,3$. We leave the details to the
reader.
\end{proof}

\begin{lemma}
\label{L6}For any $n\in 
\mathbb{N}
,n\geq 13$, define the polynomial of degree $n$, 
\begin{equation*}
M(x)=\dsum\limits_{k=0}^{n-4}(8k^{3}+60k^{2}+130k+75)\dbinom{n}{k+4}x^{k+4}%
\text{.}
\end{equation*}%
Then $M$ has exactly one root in the interval $(-1,0)$.
\end{lemma}

\begin{proof}
$M(x)=x^{4}N(x)$, where $N(x)=\dsum%
\limits_{k=0}^{n-4}(8k^{3}+60k^{2}+130k+75)\dbinom{n}{k+4}x^{k}$, which
avoids the zero of order $4$ at $x=0$. Note that the number of roots of $M$
and $N$ in $(-1,0)$ are identical. For $j\leq n-4,$%
\begin{eqnarray*}
N^{(j)}(x) &=&\dsum\limits_{k=0}^{n-4}(\allowbreak \allowbreak
8k^{3}+60k^{2}+130k+75)\dbinom{n}{k+4}(j!)\dbinom{k}{j}x^{k-j}= \\
&&(j!)\dsum\limits_{k=j}^{n-4}(8k^{3}+60k^{2}+130k+75)\dbinom{n}{k+4}\dbinom{%
k}{j}x^{k-j}\text{.}
\end{eqnarray*}

Hence $N^{(j)}(0)=(j!)(8j^{3}+60j^{2}+130j+75)\dbinom{n}{j+4}>0$, which
implies that the sequence $\left\{ N^{(j)}(0)\right\} _{j=0}^{n-4}$ has $0$
sign changes.

$N^{(j)}(-1)=(j!)\dsum\limits_{k=j}^{n-4}(\allowbreak \allowbreak
8k^{3}+60k^{2}+130k+75)\dbinom{n}{k+4}(-1)^{k-j}\dbinom{k}{j}$. We now show
that the sequence $\left\{ N^{(j)}(-1)\right\} _{j=0}^{n-4}$ has $1$ sign
change. Let 
\begin{equation*}
w(j)=n^{2}-13n-10jn+j^{2}+5j+6=\allowbreak j^{2}-5\left( 2n-1\right)
j+n^{2}-13n+6\text{.}
\end{equation*}

Then by Lemma \ref{L5}, $N^{(j)}(-1)=-\dfrac{1}{2}(j!)\left( n+j+1\right)
w(j)$, which implies that the number of sign changes in $\left\{
N^{(j)}(-1)\right\} _{j=0}^{n-4}$ equals the number of sign changes in $%
\left\{ w(j)\right\} _{j=0}^{n-4}$.

$w^{\prime }(j)=\allowbreak 2j-10n+5\leq 2(n-4)-10n+5=\allowbreak -8n-3<0$,
which implies that $w$ is decreasing for $1\leq j\leq n-4$.

$w(1)=\allowbreak n^{2}-23n+12$, which is $\left\{ 
\begin{array}{ll}
<0 & \text{if }1\leq n\leq 22 \\ 
>0 & \text{if }23\leq n%
\end{array}%
\right. $, and $w(n-4)=\allowbreak \allowbreak -8n^{2}+24n+2$, which is $%
\left\{ 
\begin{array}{ll}
>0 & \text{if }1\leq n\leq 3 \\ 
<0 & \text{if }4\leq n%
\end{array}%
\right. $.

\textbf{Case 1:} $13\leq n\leq 22$. Then $w(1)<0$ \& $w(n-4)<0$, which
implies that $w(j)<0$ for $1\leq j\leq n-4$. Then $N^{(j)}(-1)>0$.

\textbf{Case 2:} $n\geq 23$. Then $w(1)>0$ \& $w(n-4)<0$. Since $w$ is
decreasing for $1\leq j\leq n-4$, there is a $j=j_{0}$ such that $w(j)>0$
for $1\leq j\leq j_{0}$ \& $w(j)<0$ for $j_{0}+1\leq j\leq n-4$.

Then $N^{(j)}(-1)<0$ for $1\leq j\leq j_{0}$ \& $N^{(j)}(-1)>0$ for $%
j_{0}+1\leq j\leq n-4$. Since $j=0$ gives

$N(-1)=\dsum\limits_{k=0}^{n-4}(\allowbreak 8k^{3}+60k^{2}+130k+75)\dbinom{n%
}{k+4}(-1)^{k}=\allowbreak -\dfrac{1}{2}\left( n+1\right) \left(
n^{2}-13n+6\right) \allowbreak <0$ for $n\geq 13$, there is

one sign change in $N(-1),N^{\prime }(-1),N^{\prime \prime
}(-1),...,N^{(n-4)}(-1)$. By the Fourier Budan Theorem, $N$ has precisely
one real root in the interval $(-1,0)$.
\end{proof}

\begin{proof}
\textbf{Alternate Proof of }Theorem \ref{T1} when $p\in N,p\geq 13$: The
critical step in the proof of Theorem \ref{T1} was showing that $V(b)>0$ for 
$0<b<1$ and $V(b)<0$ for $b>1$. The rest of the proof is exactly the same as
in the proof of Theorem \ref{T1}. First, to show that $V(b)<0$ for $b>1$,
let $L(b)=\dfrac{V^{\prime }(b)}{b^{p-4}(p-2)(p-3)}$, where $V$ is given in (%
\ref{V}). It follows after some computation and simplification that 
\begin{gather*}
L^{(k)}(1)=-2\left( 2p-3\right) \left( 4p^{2}-12p-1\right) \left(
\prod\limits_{j=0}^{k-1}(p-j)\right) \\
+12p\left( 4p^{2}-16p+13\right) \left( \prod\limits_{j=1}^{k}(p-j)\right)
-24p(p-1)(2p-5)\left( \prod\limits_{j=2}^{k+1}(p-j)\right) \\
+16p(p-1)(p-2)\left( \prod\limits_{j=3}^{k+2}(p-j)\right) \text{.}
\end{gather*}%
Some more simplification yields $L^{(k)}(1)=-2(8k^{3}-36k^{2}+34k+3)\left(
\prod\limits_{j=0}^{k-1}(p-j)\right) $, which holds for any $p\in \Re ,p\neq
2,3$. Assume now that $p=n\in N-\left\{ 0,1,2,3\right\} $. Then one can
write $L(b)=-2\dsum\limits_{k=4}^{n}(8k^{3}-36k^{2}+34k+3)\dbinom{n}{k}%
(b-1)^{k}$ $,b\in \Re $ since $L^{(k)}(1)=0$ if $k>n$. Making a change of
variable in the summation yields, 
\begin{equation}
L(b)=-2\dsum\limits_{k=0}^{n-4}C_{k}\dbinom{n}{k+4}(b-1)^{k+4}\text{,}
\label{L}
\end{equation}%
where $C_{k}=8k^{3}+60k^{2}+130k+75$. Note that for real values of $p$ in
general the series $\dsum\limits_{k=0}^{\infty }C_{k}\dbinom{p}{k}%
(b-1)^{k+4} $ does not converge if $\left\vert b-1\right\vert >1$, which is
one of the difficulties present in using this approach for such values of $p$%
. Now 
\begin{eqnarray*}
\dfrac{V^{\prime }(b)}{(n-2)(n-3)} &=&b^{n-4}L(b)=L(b)\left(
\dsum\limits_{j=0}^{n-4}\dbinom{n-4}{j}(b-1)^{j}\right) \\
&=&\left( -2\dsum\limits_{k=0}^{n-4}C_{k}\dbinom{n}{k+4}(b-1)^{k+4}\right)
\left( \dsum\limits_{j=0}^{n-4}\dbinom{n-4}{j}(b-1)^{j}\right) \\
&=&-2\dsum\limits_{k=0}^{n-4}\dsum\limits_{j=0}^{n-4}C_{k}\dbinom{n}{k+4}%
\dbinom{n-4}{j}(b-1)^{j+k+4}\text{,}
\end{eqnarray*}

and thus 
\begin{equation*}
V^{\prime
}(b)=-2(n-2)(n-3)\dsum\limits_{k=0}^{n-4}\dsum\limits_{j=0}^{n-4}C_{k}%
\dbinom{n}{k+4}\dbinom{n-4}{j}(b-1)^{j+k+4}\text{,}
\end{equation*}%
which implies that 
\begin{equation}
V(b)=-2(n-2)(n-3)\dsum\limits_{k=0}^{n-4}\dsum\limits_{j=0}^{n-4}C_{k}%
\dbinom{n}{k+4}\dbinom{n-4}{j}\dfrac{(b-1)^{j+k+5}}{j+k+5}\text{,}
\label{V1}
\end{equation}%
which is a polynomial of degree $2n-3$ in $b$. Since $C_{k}>0$ for $n\geq 4$%
, it follows immediately from (\ref{V1}) that $V(b)<0$ for $b>1$.

Now we show that $V(b)>0$ for $0<b<1$. Since $V^{(k)}(1)=0,k=0,...,4$, and $%
V^{(5)}(1)<0$, $V$ is decreasing on some open interval containing $b=1$.
Since $V(0)=12(n+1)>0$, $V$ must have an even number of roots in $(0,1)$,
multiplicities included. If $V$ has two or more roots in $(0,1)$, it then
follows that $V^{\prime }$ also must have two or more roots in $(0,1)$. One
of those roots follows from Rolle's Theorem, and the other root follows from
the fact that $V$ must have a local maximum at $t\in (0,1)$, where $t$ is
the largest root in $(0,1)$. Since $V^{\prime }(b)=(n-2)(n-3)b^{n-4}L(b)$,
if $V$ has two or more roots in $(0,1)$, then $L$ must have two or more
roots in $(0,1)$. By (\ref{L}), $L(b)=-2M(b-1)$, where $M$ is the polynomial
from Lemma \ref{L6}. This contradicts Lemma \ref{L6}, which implies that $L$
has exactly one root in the interval $(0,1)$. Since $V$ must have an even
number of roots in $(0,1)$ and $V$ cannot have two or more roots in $(0,1)$, 
$V$ does not vanish in $(0,1)$. Since $V(0)>0$, one has $V(b)>0$ for $0<b<1$.
\end{proof}

\section{Future Research\label{S4}}

\subsection{\qquad $r=3$}

It would be nice to prove Theorem \ref{T1} for a much larger class of
functions than just certain powers of $z$. An approach along these lines
might be similar to the alternate proof of\textbf{\ }Theorem \ref{T1} given
above. Equivalent to (\ref{g}), we have 
\begin{eqnarray*}
g(x) &=&9\left( \dint\limits_{a}^{b}f\,^{\prime \prime \prime \prime }(t)%
{\large (}x-t{\large )}^{2}dt\right) \left( \dint\limits_{a}^{b}f\,^{\prime
\prime \prime \prime }(t){\large (}x-t{\large )}dt\right) \\
&&-\left( \dint\limits_{a}^{b}f\,^{\prime \prime \prime \prime }(t){\large (}%
x-t{\large )}^{3}dt\right) \left( \dint\limits_{a}^{b}f\,^{\prime \prime
\prime \prime }(t)dt\right) \text{.}
\end{eqnarray*}%
Now let $h(b)=g(a)$(this is almost identical to $V(b)$ used extensively
above). We were then able to derive the following formula: $%
h^{(k)}(a)=\dsum\limits_{j=4}^{k-1}\tfrac{\prod\limits_{l=0}^{j-2}(k-l)}{%
j\left( j-1\right) (j-4)!}{\large (}8j^{2}-(8j+1)r+2j-1{\large )}%
f^{(j)}(a)f^{(k-j+3)}(a)$. It is not too hard to show that $%
8j^{2}-(8j+1)r+2j-1<0$ for $j\geq 4$. One can then try to use the series
expansion $h(b)=\dsum\limits_{k=0}^{\infty }$ $h^{(r)}(a)(b-a)^{k}$ to
determine when $h(b)>0$ for $b>a$. Of course one has to worry about
convergence of this series. If

$\dsum\limits_{k=0}^{\infty }$ $h^{(r)}(a)(b-a)^{k}$ does converge for all $%
b>a$ and if $f^{(l)}(a)f^{(m)}(a)>0$ for all $l,m$, then one does obtain $%
h(b)<0$ for $b>a$. This should work for $f(z)=e^{z}$, say. A similar(but
more complicated) formula could then be derived for $s^{(k)}(a)$, where $%
s(b)=g(b)$(this is almost identical to $W(b)$ used above). One would then
try to show that $s(b)>0$ for $b>a$.

Of course, as noted earlier, Theorem \ref{T1} does \textbf{not }hold for all
values of $p,p\neq 0,1,2,3$. It probably fails for $0<p<3,p\neq 0,1,2$.

\qquad It is also interesting to ask which means arise amongst the class of
means given by Theorem \ref{T1}. We know that the arithmetic mean arises as
the real part of any of the nonreal roots of $f(z)=z^{4}$(and as the real
part of any of the nonreal roots of $f(z)=z^{r+1}$ for $r$ in general by
Proposition \ref{P4}). However, it is not clear, even for $r=3$, whether the
geometric or harmonic means also arise in this fashion. We believe that the
geometric and harmonic means do not appear, but have no proof of that fact.
This contrasts with the means which arise amongst the class of means $%
M_{f}^{r}(a,b)$, where $r$ is odd and $M_{f}^{r}(a,b)$ is the unique real
root of $P_{b}-P_{a}$ in $(a,b)$. In (\cite{H1}, Theorem 1.1) it was proved
that if $f(x)=x^{r/2}$, then $M_{f}^{r}(a,b)$ equals the geometric mean $%
G(a,b)=\sqrt{ab}$, and if $f(x)=x^{-1}$, then $M_{f}^{r}(a,b)$ equals the
harmonic mean $H(a,b)=\dfrac{2ab}{a+b}$.

\subsection{$r>3$}

First we make the following conjectures.

\begin{conjecture}
\label{CJ1}Suppose that $f\in C^{r+1}(0,\infty )$ with $f^{(r+1)}(x)\neq 0$
on $[a,b]$, and let $P_{c}$ denote the Taylor polynomial to $f$ of order $r$
at $x=c$. Then \textbf{at least} \textbf{one} pair of complex conjugate
roots of $P_{b}-P_{a}$ has real part lying between $a$ and $b$.
\end{conjecture}

\begin{conjecture}
\label{CJ2}Suppose that $f(z)=z^{n},n\in 
\mathbb{N}
,n\geq r+1$, and let $P_{c}$ denote the Taylor polynomial to $f$ of order $r$
at $x=c$. Then \textbf{every} pair of complex conjugate roots of $%
P_{b}-P_{a} $ has real part lying between $a$ and $b$.
\end{conjecture}

The only part of Conjecture \ref{CJ1} that we have proven so far for general 
$r$ is when $f(x)=x^{r+1}$. We have proven Conjecture \ref{CJ2} with $r=3$
in Theorem \ref{T1} with $p\in 
\mathbb{N}
-\{0,1,2,3\}$.

Conjecture \ref{CJ2} does not hold in general for $f(z)=z^{-n},n\in 
\mathbb{N}
$. The function $f(z)=\dfrac{1}{z}$ seems to be a good source of examples
for various values of $r$. This is perhaps not surprising since it was shown
in \cite{H1} that the odd order Taylor polynomials to $f(z)=\dfrac{1}{z}$
always intersect at a point whose $x$ coordinate is the harmonic mean $%
H(a,b)=\dfrac{2ab}{a+b}$. For the focus of this paper, if $r=4$, then the
two nonreal complex conjugate pairs of roots of $P_{b}-P_{a}$ have real
parts $x_{1}=\dfrac{1}{2}\left( 5+\sqrt{5}\right) \dfrac{ab\left( a+b\right) 
}{2b^{2}+(1+\sqrt{5})ab+2a^{2}}$ and $x_{2}=\allowbreak \dfrac{1}{2}\left( 5-%
\sqrt{5}\right) \dfrac{ab\left( a+b\right) }{2b^{2}-(\sqrt{5}-1)ab+2a^{2}}$.
Since $x_{1}-a=\allowbreak \dfrac{1}{2}\left( 1+\sqrt{5}\right) \dfrac{%
a\left( b-a+a\sqrt{5}\right) \left( b-a\right) }{2b^{2}+ab+ab\sqrt{5}+2a^{2}}%
\allowbreak >0$ and

$x_{1}-b=\allowbreak -\dfrac{1}{2}\dfrac{b\left( 4b+a+a\sqrt{5}\right)
\left( b-a\right) }{2b^{2}+ab+ab\sqrt{5}+2a^{2}}\allowbreak <0$, we have $%
a<x_{1}<b$ and thus $m(a,b)=\dfrac{1}{2}\left( 5+\sqrt{5}\right) \dfrac{%
ab\left( a+b\right) }{2b^{2}+(1+\sqrt{5})ab+2a^{2}}$ is a mean. However, if $%
a=1$ and $b=4$, then $x_{2}=\allowbreak 5\dfrac{5-\sqrt{5}}{19-2\sqrt{5}}<1$%
, and thus $x_{2}$ does \textbf{not }lie in $(a,b)$. For $r=5$, the two
nonreal complex conjugate pairs of roots of $P_{b}-P_{a}$ have real parts $%
x_{1}=\dfrac{(a+b)ab}{2\left( a^{2}-ab+b^{2}\right) }$ and $x_{2}=\dfrac{%
3(a+b)ab}{2\left( a^{2}+ab+b^{2}\right) }$. In a similar fashion, one can
show that $x_{1}$ does \textbf{not }lie in $(a,b)$, while $a<x_{2}<b$.
Jumping to $r=7$, one can easily show that the reals parts of two of the
nonreal complex conjugate pairs of roots of $P_{b}-P_{a}$ have real parts
lying in $(a,b)$, while the third does \textbf{not }lie in $(a,b)$.

\subsection{Nonreal Nodes}

One can try to extend some of the results in this paper to the case where $%
P_{c}$ is the Taylor polynomial to $f$ of order $r$ at $z=c$, and where $c$
can be nonreal. For example, consider $f(z)=z^{4}$, $r=3$, $a=2+4i$, and $%
b=4+2i$. A simple computation shows that $P_{2+4i}(z)-P_{4+2i}(z)=\left(
8-8i\right) \left( -z+2+2i\right) \left( -z+3+3i\right) \left(
-z+4+4i\right) $, so that the roots of $P_{2+4i}-P_{4+2i}$ are $%
z_{1}=2+2i,z_{2}=3+3i$, and $z_{3}=4+4i$. Note that $a\leq \func{Re}%
z_{j}\leq b$ and $a\leq \func{Im}z_{j}\leq b$ for $j=1,2,3$, but there is
not a strict inequality in each case. Also, $3+3i$ is the arithmetic mean of 
$a=2+4i$ and $b=4+2i$, something we saw for the case of real $a$ and $b$.

\end{document}